\documentclass{amsart}
\usepackage{mathrsfs}
\usepackage{amsfonts}
\usepackage{latexsym,amsmath,amssymb}

\usepackage{xcolor} 
\usepackage{ulem}   
\usepackage{soul}   

 \renewcommand{\epsilon}{\varepsilon}

\newtheorem{theorem}{Theorem}[section]

 \newtheorem{lemma}[theorem]{Lemma}

 \newtheorem{prop}[theorem]{Proposition}
\newtheorem{deff}[theorem]{Definition}
 \newtheorem{rem}[theorem]{Remark}
 \newcommand{\bth}{\begin{theorem}}
 \newcommand{\ble}{\begin{lemma}}
 \newcommand{\bcor}{\begin{corr}}
 \newcommand{\bdeff}{\begin{deff}}
 \newcommand{\bprop}{\begin{proposition}}
 \newcommand{\ele}{\end{lemma}}
 \newcommand{\ecor}{\end{corr}}
 \newcommand{\edeff}{\end{deff}}
 
 \newcommand{\eprop}{\end{proposition}}

 \newcommand{\la}{\lambda}

 \renewcommand{\Pi}{\varPi}

 \renewcommand{\epsilon}{\varepsilon}

  \newcommand{\R}{{\mathbb R}}

 \newcommand{\p}{\partial}

\numberwithin{equation}{section}

\newtheorem{lem}{Lemma}[section]

\thanks{The first author is supported  by NSFC Grant, No.12101145 and  Guangxi Science, Technology Project, Grant No. GuikeAD22035202. The second and last authors are supported by NSFC Grant No. 12431008. The third author is supported by NSFC Grant No.12141105.}

\title
[Hessian curvature hypersurfaces with prescribed Gauss image]{Hessian curvature hypersurfaces with prescribed Gauss image
}
\author{Rongli Huang}
\address{School of Mathematics and Statistics, Guangxi Normal University,
Guilin, Guangxi 541004, China}
 \email{ronglihuangmath@gxnu.edu.cn}

 \author{Changzheng Qu}
\address{School of Mathematics and Statistics, Ningbo University, Ningbo, China}
	\curraddr{}
	\email{quchangzheng@nbu.edu.cn}
	\thanks{}

 \author{ZhiZhang Wang}
\address{School of Mathematical Science, Fudan University, Shanghai 200433, China}
\email{zzwang@fudan.edu.cn}

\author{Weifeng Wo}
\address{School of Mathematics and Statistics, Ningbo University, Ningbo, China}
	\curraddr{}
	\email{woweifeng@nbu.edu.cn}
	\thanks{}

\begin{document}
\maketitle
\begin{abstract}

In this paper, we investigate Hessian curvature hypersurfaces with prescribed Gauss images. Given geodesically strictly convex bounded domains $\Omega$ in $\mathbb{R}^n$ and  $\tilde{\Omega}$ in the unit hemisphere,  we prove that there is a strictly convex graphic hypersurface defined in $\Omega$ with prescribed $k$-Hessian curvatures such that its Gauss image is $\tilde{\Omega}$. 
Our proof relies on a novel $C^2$ boundary estimate which utilizes the orthogonal invariance of hypersurfaces. Indeed, we employ some special vector fields generated by the infinitesimal rotations in $\mathbb{R}^{n+1}$ to establish the boundary $C^2$ estimates. This new approach enables us to handle the additional negative terms that arise when taking second order derivatives near the boundary. 

\end{abstract}

\section{Introduction}
Suppose $\mathbb{R}^{n+1}$ is the $n+1$ dimensional  Euclidean space. Always assume $\Omega$ is a strictly convex bounded domain in $\mathbb{R}^n$.
Suppose $u$ is a function defined on $\Omega$ and $M_u=\{X=(x,u(x)); x\in\Omega\}$ is a graphic hypersurface in  $\mathbb{R}^{n+1}$. Let \(N\) be the upward unit normal vector to \(M_u\). Then $N$ is the Gauss map for $M_u$ and its image lies in the unit hemisphere $$\mathbb{S}^n_+=\{x=(x_1,\cdots,x_{n+1})\in \mathbb{S}^n; x_{n+1}>0\}.$$
Assume  $\tilde{\Omega}$ is a (geodesically) strictly convex bounded domain in $\mathbb{S}^n_+$.
We can pose the second boundary value problem for the prescribed $k$-Hessian curvature equations as following:
Given $\Omega,\tilde{\Omega}$, can one  find  a positive constant $c$ and  a strictly convex graphic hypersurface $M_u$ defined by $u$ such that
\begin{equation}\label{e1.1}
\sigma_{k}(\kappa[M_u])=c \quad  \text{in} \  \Omega,
\end{equation}
 and the Gauss image of $M_u$ is $\tilde{\Omega}$?
 Here  $\kappa[M_u]=(\kappa_{1},\kappa_{2},\cdots,\kappa_n)$  are the principal curvatures of $M_u$ and
 $\sigma_{k}$ denotes the $k$-th  elementary symmetric function
\begin{equation*}
\sigma_{k}(\kappa)=\sum_{1\leq i_{1}< i_{2}<\cdots< i_{k}\leq n}\kappa_{i_{1}}\kappa_{i_{2}}\cdots\kappa_{i_{k}}.
\end{equation*}

By a projection map, one can map $\mathbb{S}^n_+$ into $\mathbb{R}^n$ which is a diffeomorphism. Therefore it maps $\tilde{\Omega}$ to some strictly convex domain $\Omega^*$ in $\mathbb{R}^n$. That is
\begin{equation}\label{bc}
Du(\Omega)=\Omega^*,
\end{equation}
where $Du$ is the gradient map. \eqref{bc} is known as the second boundary value condition.
We will explain more detail in the next section. Thus given $\tilde{\Omega}$ is equivalent to given $\Omega^*$.

The problem of finding hypersurfaces with prescribed curvatures is a classical topic in differential geometry.  Caffarelli-Nirenberg-Spruck \cite{Caffarelli1986} studied the existence of star-shaped closed hypersurfaces with prescribed Hessian curvature.  For further studies  of closed hypersurfaces related to prescribed curvature problem, please see \cite{GuanGuan2002,   GuanLi2012,Guan2015, Jiao2022,Sheng2004,Urbas2000} and the references therein.
Caffarelli-Nirenberg-Spruck \cite{Caffarelli1988} and  Ivochkina \cite{Ivochkina1989, Ivochkina1990} considered the homogeneous and nonhomogeneous Dirichlet problem for $k$-Hessian curvature equations. The Neumann boundary problems for mean curvature equation and Gauss curvature equation were investigated by Ma-Xu \cite{Ma2016} and Lions-Trudinger-Urbas \cite{Lions1986} respectively. For more results on Dirichlet or Neumann boundary problems of curvature equations, we refer to \cite{Jiao2022, Ma2018, Urbas1996}, and the reference therein.

For the second boundary value problem, Pogorelov \cite{Pogorelov1964} first studied the Monge-Amp\`ere equations with \eqref{bc}.  Caffarelli \cite{Caffarelli1996} and Urbas \cite{Urbas1997, Urbas2001} established the existence of globally smooth solutions with \eqref{bc} for Monge-Amp\`ere equations and Hessian equations. One may see \cite{Chen2021,Chen2016, Jiang2018, von2010} and the reference therein for recent progress.
 Such boundary condition has attracted increased attention, since it is connected to mass transfer problems \cite{Caffarelli1992} and minimal Lagrangian diffeomorphisms \cite{Wolfson1997, Urbas2007}. 

Second boundary value problems also appear in various geometric problems, such as prescribed curvature problem, minimal Lagrangian submanifolds, and geometric flows. Urbas \cite{Urbas2002} first studied Weingarten hypersurfaces with prescribed Gauss images. However, the general $k$-Hessian curvature equations ($k<n$) were not addressed there. Brendle-Warren \cite{Brendle2010} studied minimal Lagrangian graphs in strictly convex domains. Recent progress on curvature equations with prescribed gradient images can be found in \cite{Huang2022} and \cite{Wang2023}. Schn\"urer-Smoczyk \cite{Schnurer2002, Schnurer2003} investigated Gauss curvature flows and Weingarten curvature flows with second boundary conditions. Further research on curvature flows can be found in \cite{ Huang2015,Kitagawa2012, Wang2023, Wang2024}.

Now we state our first result:
\begin{theorem}\label{t1.1}
	Assume $\Omega$ and $\tilde{\Omega}$ are bounded, (geodesically) strictly convex domains with smooth boundaries in $\mathbb{R}^n$ and $\mathbb{S}^n_+$ respectively. Then, up to a constant, there exists a unique smooth strictly convex graphic hypersurface $M_u$ determined by a smooth function $u$ and a unique positive constant $c$ satisfying \eqref{e1.1} such that the Gauss image of $M_u$ is $\tilde{\Omega}$.
\end{theorem}
In \cite{Urbas2001}, Urbas discussed the $k$-Hessian equations in general form. Theorem 1.5 in \cite{Urbas2001} can be generalized to the corresponding prescribed curvature  equation by using our approach. Specifically, we consider the following equation:
\begin{equation}\label{1.2}
	\sigma_k^{1/k}(\kappa[M_u]) =\psi(\langle X,N\rangle, N)\quad \text{in} \ \ \Omega,
\end{equation}
where $\langle X,N\rangle$ is the support function of $M_u$ and $\psi(z,p)\in C^{\infty}(\mathbb{R}\times \tilde{\Omega})$ is a positive function satisfying the following conditions,
\begin{align}
&  \psi_z\leq 0 \text{ on } \mathbb{R}\times \tilde{\Omega};  \label{Cond1}\\
& \psi(z,p)\rightarrow +\infty, \text{ as } z\rightarrow -\infty;\ \ \  \psi(z,p)\rightarrow 0 \text{ as } z\rightarrow +\infty \label{Cond2}.
\end{align}
Thus we have the following theorem:
\begin{theorem}\label{t1.2}
Assume $\Omega$ and $\tilde{\Omega}$ are bounded, (geodesically) strictly convex domains with smooth boundaries in $\mathbb{R}^n$ and $\mathbb{S}^n_+$ respectively. Assume $\psi\in C^{\infty}(\mathbb{R}\times \tilde{\Omega})$ is a positive function satisfying \eqref{Cond1} and \eqref{Cond2}.
Then, there exists a smooth strictly convex graphic hypersurface $M_u$ determined by a smooth function $u$ satisfying \eqref{1.2} such that the Gauss image of $M_u$ is $\tilde{\Omega}$. If \eqref{Cond1} changes to  $\psi_z<0$ on $\mathbb{R}\times \tilde{\Omega}$, the solution is unique in the class of strictly convex functions.
\end{theorem}

If $\psi$ only depends on $N$, i.e. $\psi=\psi(N)$, we also have an existence result as Theorem \ref{t1.1}:
\begin{theorem}\label{t1.3}
Assume $\Omega$ and $\tilde{\Omega}$ are bounded, (geodesically) strictly convex domains with smooth boundaries in $\mathbb{R}^n$ and $\mathbb{S}^n_+$ respectively. Assume $\psi\in C^{\infty}(\tilde{\Omega})$ is a positive function.
Then, up to a constant, there exist a unique smooth strictly convex graphic hypersurface $M_u$ determined by a smooth function $u$ and a unique positive constant $c$ satisfying
\begin{equation}\label{1.3}
	\sigma_k^{1/k}(\kappa[M_u]) =c\psi(N)\quad \text{in} \ \ \Omega
\end{equation}
such that the Gauss image of $M_u$ is $\tilde{\Omega}$.
\end{theorem}
\begin{rem}
In Theorem \ref{t1.1} and \ref{t1.3}, the uniqueness is  in the class of strictly convex functions.
\end{rem}

We now highlight the novelty of our proof. To find strictly convex solutions, we consider the dual problem via the Legendre transform. In fact, the dual equation of \eqref{1.2} is a quotient Hessian equation:
\begin{equation}\label{Legendre1}
F^*\Big(w^* \sum_{k,l}b^*_{ik} (D_{kl}u^*) b^*_{lj}\Big)=\left[\frac{\sigma_n}{\sigma_{n-k}}\Big(\la^*\Big[w^*\sum_{k,l} b^*_{ik}(D_{kl}u^*)b^*_{lj}\Big]\Big)\right]^{\frac{1}{k}}=\psi^*(y,u^*),
\end{equation}
where $y=Du(x)$, $u^*$ is the Legendre transform of $u$, $F^*$ represents the dual operator and $\psi^*$
is the corresponding dual function. The terms $w^*$, $b^*_{ik}$, and $D_{kl} u^*$ come from the Legendre transform of the original functions. This is the central equation to our approach. We will  provide more details in the next section.

One  main difficulty arises when taking second derivatives near the boundary. Direct differentiation  yields additional negative terms $-C\sum_i F^{*ii}\Delta u^*$, which cannot be eliminated by simply adding a test function due to the structure of the quotient Hessian.
To overcome the difficulty, our strategy is to take derivative along some special vector fields, which are generated by infinitesimal rotations in $\mathbb{R}^{n+1}$. Here we have used the orthogonal invariance of the hypersurfaces.
Note that the method to establish the second-order estimate in \cite{Urbas2001,Urbas2002} is not applicable to our $k$-Hessian curvature equations.

The structure of the paper is as follows. In Section 2, we present preliminaries, including the notations, the three different models for representing the curvature equations via the Gauss map and the Legendre transform. In Section 3, we obtain the $C^0$ estimate  and use continuity method and degree theory to prove our main existences theorems. In Section 4, we will prove the strict obliqueness by modifying the argument in \cite{Urbas2001}. Section 5 is dedicated to studying the vector field generated by the rotations in $\mathbb{R}^{n+1}$, where we construct a special family of rotations and the corresponding vector field. Finally in Section 6, we establish the $C^2$ estimates, including global and boundary $C^2$ estimates for \eqref{Legendre1} by using the vector field constructed in Section 5.

\section{Preliminaries}\label{Pre}
Assume that $u$ is a sufficiently smooth function defined  in some bounded domain $\Omega\subset\mathbb{R}^{n}$.
Let $D$ be the connection with respect to the standard Euclidean metric. We always denote $$u_i=D_{i}u=\dfrac{\partial u}{\partial x_{i}},
u_{ij}=D_{ij}u=\dfrac{\partial^{2}u}{\partial x_{i}\partial x_{j}},
u_{ijk}=D_{ijk}u=\dfrac{\partial^{3}u}{\partial x_{i}\partial x_{j}\partial
x_{k}}, \cdots $$
and $|Du|=\sqrt{\sum_{i=1}^{n}|D_{i}u|^{2}}.$

\subsection{Vertical graph}
We follows \cite{Caffarelli1986}   to give various geometric quantities associated with $M_u$.  In the standard coordinates of $\mathbb{R}^{n+1}$, the indued metric of $M_u$ is
\begin{equation}\label{e2.2}
g_{ij}=\delta_{ij}+D_{i}uD_{j}u,\quad 1\leq i,j\leq n.
\end{equation}
While the inverse of $g_{ij}$ and second fundamental form of $M_u$ are given by
\begin{equation}\label{e2.3}
g^{ij}=\delta_{ij}-\frac{D_{i}uD_{j}u}{1+|Du|^{2}},\quad 1\leq i,j\leq n,
\end{equation}
and
\begin{equation}\label{e2.4}
h_{ij}=\frac{D_{ij}u}{\sqrt{1+|Du|^{2}}},\quad 1\leq i,j\leq n.
\end{equation}
The  upward unit normal vector to $M_u$ is expressed by
\begin{equation}\label{e2.5}
N(X)=\frac{(-Du,1)}{\sqrt{1+|Du|^{2}}}.
\end{equation}

The principal curvatures of $M_u$ are the eigenvalues of  $h^{j}_{i}\equiv h_{ik}g^{kj}$, which are the eigenvalues of the symmetric matrix, by \cite{Caffarelli1986}
\begin{equation}\label{e2.17}
a_{ij}=\frac{1}{w}b^{ik}(D_{kl}u)b^{lj},
\end{equation}
where $w=\sqrt{1+|Du|^{2}}$ and $b^{ij}=\delta_{ij}-\frac{D_{i}uD_{j}u}{w(1+w)}$ is the square root of $g^{ij}$.
Then $b_{ij}$  is the inverse of $b^{ij}$ expressed as
$b_{ij}=\delta_{ij}+\frac{D_{i}uD_{j}u}{1+w}$.

Let $\mathcal{S}$ be the vector space of $n\times n$ symmetric matrices and
\[\mathcal{S}_+=\{A\in \mathcal{S}: \lambda(A)\in \Gamma_n\},\]
where $\Gamma_n:=\{\lambda\in\R^n:\,\,\mbox{each component $\lambda_i>0$}\}$ is the convex cone, and $\lambda(A)=(\lambda_1, \cdots, \lambda_n)$ denotes the eigenvalues of $A.$
Define a function $F$ by
\[F(A)=\sigma^{\frac{1}{k}}_k(\lambda(A)),\,\, A\in\mathcal{S}_+,\]
then \eqref{1.2} can be written as
\begin{equation}\label{main equation graph}
F\left(\frac{1}{w}b^{ik}(D_{kl}u)b^{lj}\right)=\psi(\langle X,N\rangle, N).
\end{equation}
Note that, in fact the function $F$ is well defined on $\mathcal{S}_k=\{A\in \mathcal{S}: \lambda(A)\in \Gamma_k\},$ where
$\Gamma_k$ is the G{\aa}rding's cone (see \cite{Caffarelli1985}).
However, in this paper, we only study strictly convex hypersurfaces, thus we restrict ourselves to $\mathcal{S}_+.$
Throughout this paper we denote
\[F^{ij}(A)=\frac{\partial F}{\partial a_{ij}}(A),\,\,F^{ij, kl}=\frac{\partial^2 F}{\partial a_{ij}\partial a_{kl}}.\]

\subsection{The Gauss map}
Let $M_u$ be an entire, strictly convex hypersurface, $N(X)$ be the upward unit normal vector to $M_u$
at $X.$ We define the Gauss map:
\[G: M_u\rightarrow \mathbb{S}^n_+;\,\, X\mapsto N(X).\]
If we take the hyperplane $\mathbb{P}:=\{X=(x_1, \cdots, x_{n}, x_{n+1}) |\, x_{n+1}=1\}$ and consider the projection of
$\mathbb{S}^n_+$ from the origin into $\mathbb{P}.$ Then $\mathbb{S}^n_+$ is mapped in
a one-to-one fashion onto $\mathbb{R}^n$. The map
$P$ is given by
\begin{equation}\label{proj}
P: \mathbb{S}^n_+\rightarrow \mathbb{R}^n;\,\,(x_1, \cdots, x_{n+1})\mapsto (y_1, \cdots, y_n),
\end{equation}
where $x_{n+1}=\sqrt{1-x_1^2-\cdots-x_n^2},$ $y_i=-\frac{x_i}{x_{n+1}}.$
Thus it is easy to check that
\begin{eqnarray}
P\circ G:& (x,u(x))\in M_u\;\; \rightarrow\;\; D u(x)\in \mathbb{R}^n.
\end{eqnarray}
Thus the gradient map $D u$ is equivalent to the Gauss map.

Next let's consider the support function of $M_u.$ We denote
\[v:=-\langle X, N\rangle=\frac{1}{\sqrt{1+|Du|^2}}\left(\sum_ix_i\frac{\partial u}{\partial x_i}-u\right).\]
Here $\langle \cdot, \cdot\rangle$ denotes the standard Euclidean inner product.
Let $\{e_1, \cdots, e_n\}$ be an orthonormal frame on $\mathbb{S}^n$ and $\nabla$ be  the standard Levi-Civita connection of $\mathbb{S}^n$.
We denote
\begin{equation}\label{sphere}
\Lambda_{ij}=\nabla_{ij}v+v\delta_{ij}
\end{equation}
to be the spherical Hessian. Here $\nabla_iv, \nabla_{ij}v$ denote the first and second covariant derivatives with respect to the standard spherical  metric.
In convex geometry, it is well known that
$$\nabla_iv=-\langle X, e_i\rangle, \ \ X=-\sum_i(\nabla_iv)e_i-vN.$$
Moreover we have
\begin{eqnarray*}
g_{ij}=\sum_k\Lambda_{ik}\Lambda_{kj},\ \ h_{ij}=\Lambda_{ij}.
 \end{eqnarray*}
This implies that the eigenvalues of the spherical Hessian are the curvature radius of $M_u$. That is,
if the principal curvatures of $M_u$ are $(\la_1, \cdots, \la_n),$ then the eigenvalues of the spherical Hessian
are $\left(\la_1^{-1}, \cdots, \la_n^{-1}\right).$
Therefore, equation \eqref{1.2} can be written as
\begin{equation}\label{main equation hyperbolic}
F^*(\nabla_{ij}v+v\delta_{ij})=\tilde{\psi}(x,v),\;\;x\in\mathbb{S}^n_+,
\end{equation}
where $F^*(A)=\left[\frac{\sigma_n}{\sigma_{n-k}}(\lambda(A))\right]^{\frac{1}{k}}$ and
\begin{equation}\label{tildepsi}
\tilde{\psi}(x,v)=\frac{1}{\psi(v, x)}.
\end{equation}

\subsection{The Legendre transform}

Suppose $M_u$ is an entire, strictly convex graphic hypersurface defined by a function $u$.
Then we have
\[x_{n+1}=\left<X, \mathbf{E}\right>=u(x_1, \cdots, x_n),\]
where $\mathbf{E}=(0, \cdots, 0, 1).$
Introduce the Legendre transform
\[y_i=\frac{\p u}{\p x_i},\,\, u^*=\sum_{i=1}^n x_iy_i-u.\]
From the theory of convex bodies, we know that
\[\Omega^*=Du(\Omega)\]
is a convex domain. It is well known that
$$\left(\frac{\p^2 u}{\p x_i\p x_j}\right)=\left(\frac{\p^2  u^*}{\p y_i\p y_j}\right)^{-1}.$$
We have, using the coordinate $(y_1,y_2,\cdots,y_n)$, the first and the second fundamental forms can be rewritten as:
$$g_{ij}=\delta_{ij}+y_iy_j, \text{ and\,\,  } h_{ij}=\frac{u^{* ij}}{\sqrt{1+|y|^2}},$$
where $\left(u^{* ij}\right)$ denotes the inverse matrix of $(u^*_{ij})$ and $|y|^2=\sum_iy_i^2$. Now let $W$ denote the Weingarten matrix of $M_u,$ then
$$(W^{-1})_{ij}=\sqrt{1+|y|^2}\sum_kg_{ik}u^*_{kj}.$$

From the discussion above, we can see that if $M_u=\{(x, u(x)) | x\in\R^n\}$ is an entire, strictly convex
hypersurface satisfying \eqref{main equation graph}, then the Legendre transform of
$u$ denoted by $u^*$ satisfies
\begin{equation}\label{Legendre}
F^*\Big(w^* \sum_{k,l}b^*_{ik} (D_{kl}u^*) b^*_{lj}\Big)=\left[\frac{\sigma_n}{\sigma_{n-k}}\Big(\la^*\Big[w^*\sum_{k,l} b^*_{ik}(D_{kl}u^*)b^*_{lj}\Big]\Big)\right]^{\frac{1}{k}}=\psi^*(y,u^*).
\end{equation}
Here $w^*=\sqrt{1+|y|^2}$, $b^*_{ij}=\delta_{ij}+\frac{y_iy_j}{1+w^*}$ is the square root of the matrix $g_{ij},$ and
\begin{equation}\label{psi*}
\psi^*(y,u^*)=\frac{1}{\psi \left(\frac{u^*}{w^*}, \frac{(-y,1)}{w^*} \right)}.
\end{equation}

\section{$C^0$ estimates  and existence }
In this section, we prove the main theorems. The existence of solutions is established by using the continuity method, the Leray-Schauder degree theory and the a prior $C^0-C^2$ estimate.  The $C^0$ estimate will be carried out here. The $C^1$ estimate follows from the boundedness of  $\Omega$ and $\Omega^*$. The $C^2$ estimate, being the most challenging, is proved in the next three sections.

{\bf The proof of Theorem \ref{t1.2}}: Suppose $f(\lambda)$ is a symmetric concavity function, where $\lambda=(\lambda_1,\cdots,\lambda_n)$. By the concavity of $f$, we have
$$f(\lambda)\leq f(1,\cdots,1)+\sum_i\frac{\p f}{\p \lambda_i}(1,\cdots,1) (\lambda_i-1).$$
Since $\sigma_k^{1/k}$ and $\left(\frac{\sigma_n}{\sigma_{n-k}}\right)^{1/k}$ are concave functions, it implies that
\begin{equation}\label{ineq:F}
F\left(\frac{1}{w}b^{ik}(D_{kl}u)b^{lj}\right)\leq C\sigma_1\left(\frac{1}{w}b^{ik}(D_{kl}u)b^{lj}\right)+C
\end{equation}
and
\begin{equation}\label{ineq:F*}
F^*\left(w^* \sum_{k,l}b^*_{ik} (D_{kl}u^*) b^*_{lj}\right)\leq C\sigma_1\left(w^* \sum_{k,l}b^*_{ik} (D_{kl}u^*) b^*_{lj}\right)+C.
\end{equation}
Therefore integrating  \eqref{ineq:F}, we get
$$\int_{\Omega}\psi(\langle X,N\rangle, N)\leq C|\Omega|+\int_{\Omega}D\left(\frac{Du}{w}\right)=C|\Omega|-\int_{\p\Omega}\frac{D_{\nu}u}{w}\leq C $$
by the the boundedness of $\Omega^*$. Here $\nu$ is the unit interior normal of $\Omega$. By \eqref{Cond2}, there exists a point $p \in \bar{\Omega}$
such that

$$\langle X,N\rangle \geq -C,$$
which implies $u(p)\leq C$. Since $Du$ is bounded, we get
\begin{equation}\label{sup}
\sup_{\Omega}u\leq C.
\end{equation}
Integrating  \eqref{ineq:F*}, we get
\begin{eqnarray*}
\int_{\Omega^*}\psi^*(y,u^*)&\leq& C|\Omega^*|+\int_{\Omega^*}\sum_{i,j}w^*g_{ij}u^*_{ij}\\
&=&C|\Omega^*|-\int_{\Omega^*}\sum_{i,j}(w^*g_{ij})_ju^*_i-\int_{\p\Omega^*}\sum_{i,j}w^*g_{ij}\nu_j^*u^*_i\leq C
\end{eqnarray*}
by the boundedness  of $\Omega$. Here $\nu^*$ is the unit interior normal of $\Omega^*$. By the definition  of $\psi^*(y,z)$, we have
\begin{align*}
&  \psi^*_z\geq 0 \text{ on } \Omega^*\times \mathbb{R}, \\
& \psi^*(y, z)\rightarrow +\infty, \text{ as } z\rightarrow +\infty,\\
&  \psi^*(y,z)\rightarrow 0 \text{ as } z\rightarrow 0 .
\end{align*}
Therefore same argument as before shows
$$\sup_{\Omega^*} u^*\leq C.$$ Since $u^*(y)=x\cdot y-u(x)$, we obtain the uniformly lower bound of $u$
\begin{equation}\label{inf}
\inf_{\Omega} u\geq -C.
\end{equation}

Combining \eqref{sup} and \eqref{inf}, we have the $C^0$ estimate of $u$. The $C^1$ estimate comes from the boundedness of $\Omega,\Omega^*$. In the following Sections 4-6, we will establish the $C^2$ estimate of \eqref{Legendre}. By Newton-Maclruing inequality, the upper bound of \eqref{Legendre} will give the uniformly  positive lower bound.
Consequently, applying the continuity method, the Leray-Schauder degree theory and the Hopf's boundary point lemma, we demonstrate the existence and uniqueness of \eqref{1.2}. This completes the proof of Theorem \ref{t1.2}.

\bigskip

Then we turn to prove Theorem \ref{t1.1} and \ref{t1.3}. If $\psi$ is a constant function in \eqref{1.3}, then \eqref{1.3} becomes  \eqref{e1.1}. Thus, we only need to prove Theorem \ref{t1.3}.

{\bf The proof of Theorem \ref{t1.3}}: We consider the approximate equation:
\begin{equation}\label{appro}
\sigma_k^{1/k}(\kappa[M_u])=e^{-\epsilon\frac{\langle X,N\rangle}{\langle N, \mathbf{E}\rangle}}\psi(N),
\end{equation}
for any positive constant $\epsilon$ with $Du_{\epsilon}(\Omega)=\Omega^*$. Here $\mathbf{E}=(0, \cdots, 0, 1).$ The right hand side obviously satisfies conditions \eqref{Cond1} and \eqref{Cond2}.
Therefore according to Theorem \ref{t1.2}, we can find a solution $u_{\epsilon}$ of \eqref{appro}. In view of the $C^0$ estimate in the proof of Theorem \ref{t1.2}, we get
\begin{equation}\label{eu}
|\epsilon u_{\epsilon}|\leq C.
\end{equation}
The $C^1$ estimate is straightforward.  For the $C^2$ estimate,  it does not depend on $C^0$ norm of $u_{\epsilon}$. Thus, if we let
\begin{equation}\label{hatu}
\hat{u}_{\epsilon}=u_{\epsilon}-\frac{1}{|\Omega|}\int_{\Omega}u_{\epsilon},
\end{equation}
 then we have
$$\|\hat{u}_{\epsilon}\|_{C^2(\bar{\Omega})}\leq C.$$
By Evance-Krylov theory, we get
$$\|\hat{u}_{\epsilon}\|_{C^{2,\alpha}(\bar{\Omega})}\leq C$$ for some $0<\alpha<1$. Thus we can assume
$$\lim_{\epsilon\rightarrow 0}\hat{u}_{\epsilon}=\hat{u},$$ in the space $C^{2,\beta}(\bar{\Omega})$ for $\beta<\alpha$. Thus $\epsilon \hat{u}_{\epsilon}\rightarrow 0$.
Using \eqref{hatu} and \eqref{eu}, we get
$$\frac{\epsilon}{|\Omega|}\left|\int_{\Omega}u_{\epsilon}\right|\leq C.$$
Thus taking some subsequence, we can prove $\epsilon u_{\epsilon}$ convergences to some constant $\log c$. Now letting $\epsilon\rightarrow 0$ in \eqref{appro}, we get
\eqref{1.3}. This completes the proof of Theorem \ref{t1.3}.

\section{Strict obliqueness}

In this section, we will prove the strict obliqueness of the problem \eqref{e1.1}.  Since our argument is closed to \cite{Urbas2002}, \cite{ Urbas1997} and \cite{Urbas2001},
we just provide a brief outline.
\begin{deff}
Suppose $\Omega$ is a strictly convex bounded domain.
A uniformly concave function $h:\mathbb{R}^n\rightarrow\mathbb{R}$ is called the defining function of $\Omega$, if
$\Omega=\{p\in\mathbb{R}^{n} : h(p)>0\}$ and  $|Dh|=1$ on $\partial \Omega$.
\end{deff}

\begin{lemma}\label{3.4}
 Suppose $\Omega$, $\Omega^*$ are bounded strictly convex domains with smooth boundary in $\mathbb{R}^{n}$. Suppose  $h(p_1,\cdots,p_n)$ is the uniform concave defining function of $\Omega^*$,
then $$\nu^*(p)=Dh(p_1,\cdots,p_n)=\left(\frac{\p h}{\p p_1},\cdots, \frac{\p h}{\p p_n}\right)$$ is the interior unit normal of $\partial\Omega^*$. Let $\nu$ be the unit interior normal of $\partial\Omega$.
 If $u$ is a strictly convex solution to  (\ref{e1.1}) with $\Omega^*=Du(\Omega)$,  then the strict obliqueness estimate
\begin{equation}\label{ee3.9}
\langle\nu^*(Du(x)), \nu(x)\rangle\geq c_0>0
\end{equation}
holds for any $x\in\partial \Omega$. Here $c_0$ is some constant only depending on  $\Omega$ and $\Omega^*$.
\end{lemma}
\begin{proof}
Let $$\chi(x)=\langle\nu^*(Du(x)),\nu(x)\rangle=\sum_kh_{p_k}(Du(x))\nu_k(x),$$ where $h_{p_k}=\frac{\p h}{\p p_k}$ and $\nu=(\nu_1,\cdots,\nu_n)$.
A same argument as \cite{Urbas2001},  we get $\chi\geq 0$ on $\p \Omega$ and we can prove
\begin{eqnarray}\label{a}
\chi=\sqrt{u^{ij}\nu_i\nu_j(D_{kl}u) h_{p_k}h_{p_l}}  \ \ \text{ on }\ \  \p\Omega.
\end{eqnarray}
Suppose $x_0\in\p\Omega$ is the minimum value point of the function $\chi |_{\p\Omega}$. Let
$$\omega(x)=\chi(x)+A h(Du(x)),$$
where $A$ is a positive sufficiently large constant. A same argument as \cite{Urbas2001}, to prove that $\sum_{k,l}(D_{kl}u) h_{p_k}h_{p_l}$ has a uniform positive lower bound, we only need to prove
\begin{eqnarray}\label{a1}
\omega_n(x_0)\geq -C,
\end{eqnarray}
where $n$ is the unit outward normal of $\p\Omega$.

Suppose $h^*(x)$ is the uniformly concave function of $\Omega$  giving $\nu = Dh^*$. Let $u^*$ be the  Legendre transform of $u$ satisfying \eqref{Legendre}.  We define
$$\chi^*(y)=\langle\nu^*(y),\nu(Du^*(y))\rangle, $$ then $\chi^*(y)=\chi(x)$, where $y=Du(x)$. Let
$$\omega^*(y)=\chi^*(y)+A^* h^*(Du^*(y)).$$ Similarly, to prove $u^{ij}\nu_i\nu_j$ has a uniform positive lower bound, we only need to prove
\begin{eqnarray}\label{a2}
\omega^*_{n^*}(y_0)\geq -C,
\end{eqnarray}
where $n^*$ is the unit outward normal of $\p\Omega^*$ and $y_0=Du(x_0)$. To obtain \eqref{ee3.9}, 
it suffices to prove \eqref{a1} and \eqref{a2}.

At first we give the proof of \eqref{a1}. Denote
$$G(Du,D^2u)=\sigma_k(a_{ij}) \text{ and } \mathcal{L}f=G^{ij}D_{ij}f+(G^s-\psi^s)D_sf$$
for any function $f$, where
$$G^{ij}=\frac{\p G}{\p u_{ij}}=\sum_{p,q}\sigma_k^{pq}\frac{1}{w}b^{ip}b^{qj},$$
and
\begin{eqnarray*}
G^s&=&\frac{\p G}{\p u_s}=-\frac{u_s}{w} \sum_{i,j}G^{ij}u_{ij}-\frac{2}{w(1 + w)}\sum_{t,j}
\sigma^{ij}_k a_{it}( wu_t b^{sj} + u_jb^{ts}),\\
\psi^s&=&\frac{\p \psi}{\p u_s}=\psi_z\left(\frac{x_s}{w}-\frac{\langle x,Du\rangle-u}{w^3}u_s\right)+\sum_{i\neq n+1}\psi_{p_i}\left(\frac{-\delta_{is}}{w}+\frac{u_iu_s}{w^3}\right)-\psi_{p_{n+1}}\frac{u_s}{w^3}.
\end{eqnarray*}
Near $x_0$, we can calculate
\begin{equation*}\label{e3.10}
 \begin{aligned}
\mathcal{L}\omega=&\sum_{l,m,k}G^{ij}u_{il}u_{jm}(h_{p_{k}p_{l}p_{m}}\nu_{k}+A h_{p_{l}p_{m}})\\
&+2\sum_{l,k}G^{ij}h_{p_{k}p_{l}}u_{li}\nu_{kj}+\sum_kG^{ij}h_{p_{k}}\nu_{kij}+\sum_kG^{s}h_{p_{k}}\nu_{ks}\\
\leq& \sum_{l,m,k}(h_{p_{k}p_{l}p_{m}}\nu_{k}+A h_{p_{l}p_{m}}+\delta_{lm})G^{ij}u_{il}u_{jm}+C\mathcal{T}+C,
 \end{aligned}
\end{equation*}
where $\mathcal{T}=\sum_iG^{ii}$. Since $D^{2}h\leq-\theta I$ for some positive constant $\theta$, we may choose sufficiently large $A$ satisfying
$$(h_{p_{k}p_{l}p_{m}}\nu_{k}+A h_{p_{l}p_{m}}+\delta_{lm})<0,$$
which implies
\begin{equation}\label{e3.11}
\mathcal{L}\omega\leq C\mathcal{T} \    \text{ in } \   \ \Omega,
\end{equation}
by $\mathcal{T}\geq c_1>0$ for some constant $c_1$.

Without loss of generality, we may assume $x_{0}$ is the origin and that the positive $x_{n}$-axis is in the interior normal direction to $\partial\Omega$
at the origin. Suppose near the origin, the boundary $\partial\Omega$ is given by
\begin{equation}\label{e3.12}
x_{n}=\rho(x')=\frac{1}{2}\sum_{\alpha<n}\kappa^{b}_{\alpha}x^{2}_{\alpha}+O(|x'|^{3}),
\end{equation}
where $\kappa^{b}_{1},\kappa^{b}_{2},\cdots,\kappa^{b}_{n-1}$ are the principal curvatures of $\partial\Omega$  at the origin and $x'=(x_{1},x_{2},\cdots,$ $x_{n-1})$.

Denote a neighborhood of $x_{0}$ in $\Omega$ by
$$\Omega_{r}:=\{x\in\Omega:\rho(x')< x_{n}<\rho(x')+r^{2}, |x'|<r\},$$
where $r$ is a small positive constant to be chosen.
Define a barrier function
\begin{equation}\label{boundary}
\varphi(x)=-\rho(x')+ x_{n}+\delta|x'|^{2}-Kx^{2}_{n},
\end{equation}
where $\delta=\frac{1}{6}\min\{\kappa^{b}_{1},\kappa^{b}_{2},\cdots,\kappa^{b}_{n-1}\}$ and $K$ is some large undetermined positive constant. It is clear that $-\varphi$ is strictly convex when $r$ is sufficiently small. By the concavity of $\sigma_k^{1/k}$ and Proposition 2.1 in \cite{Jiao2022}, one can show that
\begin{equation*}
\begin{aligned}
G^{ij}\varphi_{ij}\leq& -kG^{1-1/k}(Du,D^2u)G^{1/k}(-D^2\varphi, Du)\\
\leq& -kG^{1-1/k}(Du,D^2u)\frac{\sigma_k^{1/k}(-D^2\varphi)}{w^{1+2/k}}\\
\leq& -k\eta_0G^{1-1/k}(Du,D^2u)\frac{K^{1/k}}{w^{1+2/k}},
\end{aligned}
\end{equation*}
where $\eta_0$ is some small positive constant. Thus, since $|G^s|$ is bounded, we get
\begin{equation}\label{e2.45}
  \mathcal{L}\varphi
  \leq -c_3\mathcal{T},\,\,\,x\in\Omega_{r},
\end{equation}
when $K$ is sufficiently large and $r$ is sufficiently small.

Note that the boundary $\partial \Omega_r$ consists of three parts: $\partial \Omega_r
= \partial_1 \Omega_r \cup \partial_2 \Omega_r \cup \partial_3 \Omega_r$, where
$\partial_1 \Omega_r$ and $\partial_2 \Omega_r$ are defined by  $\{x_n=\rho\}\cap\bar{\Omega}_r$ and $\{x_n=\rho +r^2\}\cap\bar{\Omega}_r$
respectively, and $\partial_3 \Omega_r$ is defined by $\{|x'| = r\}\cap\bar{\Omega}_r$. Thus,  when $r$ is sufficiently small (depending on $\delta$ and $K$), we have
\begin{equation}
\label{BC2-12}
\begin{aligned}
\varphi \geq &  \frac{\delta}{2} |x'|^2, \mbox{ on } \partial_1 \Omega_r,\\
\varphi \geq & \frac{r^2}{2},\ \ \ \ \mbox{ on } \partial_2 \Omega_r,\\
\varphi \geq & \frac{\delta r^2}{2}, \ \ \ \mbox{ on } \partial_3 \Omega_r.
\end{aligned}
\end{equation}
Therefore we have
$$\varphi\geq 0, \,\,\,\text{on } \partial\Omega_{r}.$$
In order to obtain the desired results, it suffices to consider the auxiliary function
$$\Phi(x)=\omega(x)-\omega(x_{0})+B\varphi(x),$$
where $B$ is some positive large constant to be determined.
Combining (\ref{e3.11}) with (\ref{e2.45}), a direct computation yields
\begin{equation*}
\mathcal{L}(\Phi(x))\leq (C-B c_{3})\mathcal{T}<0,
\end{equation*}
when $B$ is large.
On $\partial_1\Omega$, it is clear that $\Phi \geq0$. On $\partial_2 \Omega_r \cup \partial_3 \Omega_r$,  we also have $\Phi\geq 0$ if $B$ is sufficiently large.
By the maximum principle, we arrive at
\begin{equation}\label{e3.15aaaa}
\Phi(x)\geq 0  \ \ \text{ on } \bar{\Omega}_r.
\end{equation}
Combining with $\Phi(x_0)=0$, we have $\partial_n\Phi(x_0)\geq 0$, which gives \eqref{a1}.

Finally, we turn to prove \eqref{a2}, which is similar to the proof of \eqref{a1}. Define
$$\mathcal{L}^*=G^{*ij}D_{ij},$$
where $$G^*(y, D^2 u^*)=\frac{\sigma_n}{\sigma_{n-k}}\Big(\la^*\Big[w^* \sum_{k,l}b^*_{ik}u^*_{kl}b^*_{lj}\Big]\Big)\text{ and } G^{*ij}=\frac{\p G^*}{\p u^*_{ij}}.$$
Here, $w^*,b^*_{ij}$ are defined in subsection 2.3. Similar calculation as above shows
\begin{equation}\label{e3.16}
\mathcal{L^*}\omega^*\leq C\mathcal{T}^*,
\end{equation}
where $\mathcal{T}^*=\sum_iG^{*ii}$.

Since our equation is invariant under rotations in $\mathbb{R}^n$.
We may assume the positive $y_{n}$-axis is the interior normal direction to $\partial\Omega^*$ at $y_0=((y_0)_1,(y_0)_2,\cdots,(y_0)_n)$.
Suppose near $y_0$, the boundary $\partial\Omega^*$ is given by
\begin{equation}\label{e3.12n}
\bar{y}_n=y_{n}-(y_0)_n=\rho^*(y')=\frac{1}{2}\sum_{\alpha<n}\kappa^{b*}_{\alpha}(y_{\alpha}-(y_0)_{\alpha})^2+O(|y'|^{3}),
\end{equation}
where $\kappa^{b*}_{1},\kappa^{b*}_{2},\cdots,\kappa^{b*}_{n-1}$ are the principal curvatures of $\partial\Omega^*$  at $y_0$ and $y'=(y_{1}-(y_0)_1,y_{2}-(y_0)_2,\cdots,y_{n-1}-(y_0)_{n-1})$.
Denote a neighborhood of $y_{0}$ in $\Omega^*$ by
\begin{equation}\label{Omega*}
\Omega^*_{r}:=\{y\in\Omega^*:\rho^*(y')< \bar{y}_{n}<\rho^*(y')+r^{2}, |y'|<r\},
\end{equation}
where $r$ is a small positive constant to be chosen.
Define a barrier function
\begin{equation}\label{boundary*}
\varphi^*(y)=-\rho^*(y')+ \bar{y}_{n}+\delta^*|y'|^{2}-K^*\bar{y}_{n}^2,
\end{equation}
where $\delta^*=\frac{1}{6}\min\{\kappa^{b*}_{1},\kappa^{b}_{2*},\cdots,\kappa^{b*}_{n-1}\}$ and $K^*$ is some large undetermined positive constant. Then we have
\begin{equation*}
  \mathcal{L^*}\varphi^*
  \leq -c_4\mathcal{T}^*,\,\,\,y\in\Omega^*_{r}, \ \ \text{ and } \varphi^*\geq 0 \text{ on } \p\Omega_r^*.
\end{equation*}
Therefore,  for the  auxiliary function
$$\Psi(y)=\omega^*(y)-\omega^*(y_{0})+B^*\varphi^*(y),$$
we have $\mathcal{L}^*\Psi\leq 0$ if $B^*$ is sufficiently large, which implies it is
always nonnegative on $\bar{\Omega}_r^*$ and $\Psi(y_0)=0$.
Thus, we get \eqref{a2}.
\end{proof}

\section{The derivatives from rotations}
First, let's prove a technical lemma. We follow the notations introduced in Section \ref{Pre}.
\begin{lem}\label{4.1}
Suppose $(y_1,\cdots,y_n)$ are the rectangular coordinate of $\mathbb{R}^n$. Suppose $\nabla$ is the standard Levi-Civita connection on the unit sphere. Let
	\[
	e_i = \sum_kw^* \, b_{ik}^*  \frac{\partial}{\partial y_k}.
	\]
Then $\{e_1,\cdots,e_n\}$ is an orthonormal frame on $\mathbb{S}^n$. Moreover, for any function $v$ defined on some subset of $\mathbb{R}^n$, we get
\begin{eqnarray}\label{form}
\sum_{k,l}w^* \, b_{ik}^* D_{kl}v \, b_{lj}^*=\nabla^2_{ij}\frac{v}{w^*}+\frac{v}{w^*}\delta_{ij}.
\end{eqnarray}
	\end{lem}
	\begin{proof}
Denote $g_{ij} = \delta_{ij} + y_i y_j$
and $g^{ij} = \delta_{ij} - \frac{y_i y_j}{1+|y|^2}$. We find that the induced metric of the unit sphere on $\mathbb{R}^n$ by the map $P$ is given by
\[
\tilde{g} = \sum_{i,j}\frac{1}{1+|y|^2}\Big(\delta_{ij} - \frac{y_i y_j}{1+|y|^2}\Big) d y_i \otimes d y_j
   = \frac{g^{ij}}{{w^*}^2} d y_i \otimes d y_j.
\]
Thus, we get
$$\tilde{g}(e_i,e_j)=g^{pq}b^*_{ip}b^*_{qj}=\delta_{ij}.$$
Now we calculate the Christoffel symbol $\Gamma_{ij}^k$ of $\tilde{g}$. First we have
\[
\frac{\partial \tilde{g}_{li}}{\partial y_j}
  = - \frac{2 y_j g^{li}}{{w^*}^4} + \frac{1}{{w^*}^2} \Big(\frac{- \delta_{lj}y_i - \delta_{ij} y_l}{{w^*}^2}
     + \frac{2 y_l y_i y_j}{{w^*}^4}\Big).
\]
Therefore we get
\[
\begin{aligned}
\Gamma_{ij}^k = \,& \frac{1}{2} \tilde{g}^{kl} \Big(\frac{\partial \tilde{g}_{li}}{\partial y_j}
  + \frac{\partial \tilde{g}_{lj}}{\partial y_i} - \frac{\partial \tilde{g}_{ij}}{\partial y_l}\Big)=  - \frac{1}{{w^*}^2} (y_i \delta_{kj} + y_j \delta_{ki}).
\end{aligned}
\]
Denote $\tilde{v} = \frac{v}{w^*}$.  Then we have 
\[
\frac{\partial \tilde{v}}{\partial y_k} = \frac{1}{w^*} \frac{\partial v}{\partial y_k}
  - \frac{v y_k}{{w^*}^3}
\]
and
\[
\frac{\partial^2 \tilde{v}}{\partial y_k y_l}
  = \frac{1}{w^*} \frac{\partial^2 v}{\partial y_k y_l} - \frac{y_k}{{w^*}^3} \frac{\partial v}{\partial y_l}
     - \frac{y_l}{{w^*}^3} \frac{\partial v}{\partial y_k} - \frac{\delta_{kl} v}{{w^*}^3}
        + 3 \frac{y_k y_l}{{w^*}^5} v.
\]
Consequently, we have
\[
\begin{aligned}
\nabla_{ij} \tilde{v} = \,& \sum_{k,l}{w^*}^2 b^*_{ik} b^*_{jl} \nabla_{k l} \tilde{v}\\
 = \,& \sum_{k,l}{w^*}^2 b^*_{ik} b^*_{jl} \Big(\frac{\partial^2 \tilde{v}}{\partial y_k y_l}
     - \Gamma_{kl}^s \frac{\partial \tilde{v}}{\partial y_s}\Big)\\
 = \,& \sum_{k,l}{w^*}^2 b^*_{ik} b^*_{jl} \Big[\frac{v_{kl}}{w^*} - \frac{y_k v_l}{{w^*}^3}
     - \frac{y_l v_k}{{w^*}^3} - \frac{\delta_{kl} v}{{w^*}^3}\\
      &  + 3 \frac{y_k y_l v}{{w^*}^5}
        +\sum_s \frac{y_k \delta_{sl} + y_l \delta_{sk}}{{w^*}^2} \Big(\frac{v_s}{w^*}
  - \frac{v y_s}{{w^*}^3}\Big)\Big]\\
  = \,& \sum_{k,l}w^* b^*_{ik} b^*_{jl} v_{kl} - \frac{v}{w^*} \delta_{ij}.
\end{aligned}
\]
Here, in the first equality, $\nabla_{ij}$ in the left hand side means taking covariant derivatives with respect to $e_i,e_j$ and
$\nabla_{kl}$ in the right hand side means taking covariant derivatives with respect to $\frac{\p}{\p y_k},\frac{\p}{\p y_l}$.
Thus, \eqref{form} follows immediately.
\end{proof}

Always assume $\epsilon_0$ is a small positive constant. Suppose $\{A_t\}, t\in[0,\epsilon_0]$ is a family of orthogonal matrices in $\mathbb{R}^{n+1}$. For any $y\in\mathbb{R}^n$, we define
\begin{eqnarray}\label{sigmat}
\sigma_t(y)=PA_t^{-1}P^{-1}(y),
\end{eqnarray}
where $P: \mathbb{S}^n_+\rightarrow \mathbb{R}^n$ is the projection map defined by \eqref{proj}, and $P^{-1}, A_t^{-1}$ are inverse maps of $P, A_t$ respectively.
 Suppose $u^*$ is a solution of \eqref{Legendre}. We then define
 \begin{equation}\label{ust}
 \tilde{u}_t^*(y)=\frac{\sqrt{1+|y|^2}}{\sqrt{1+|\sigma_t y|^2}}u^*(\sigma_t(y)), \text{ and } u_t^*=u^*(\sigma_t(y)).
 \end{equation}
Then we derive the equation of $\tilde{u}^*_t$ as the following lemma.
 \begin{lem}
 For any $t\in[0,\epsilon_0]$, $\tilde{u}_t^*$  satisfies
\begin{equation}\label{Fst}
	F^*\left( \sum_{k,l}w^* \, b_{ik}^* (\tilde{u}^*_t)_{kl} \, b_{lj}^* \right) (y)= \psi^*(\sigma_t y, u^*_t(y)).
	\end{equation}
  \end{lem}
\begin{proof}

For the graphic hypersurface $M_u$, suppose its position vector is $X=(\xi,u(\xi))$, where $\xi\in\mathbb{R}^n$.
Suppose $x\in \mathbb{S}^n_+$ is its unit normal. We can view $X$ as a vector depending on $x$. Then the support function of $M_u$ is defined by
$$v(x)=-\langle X(x), x\rangle.$$
Let $X_t=A_t(X)$, where $A_t$ is an orthogonal  matrix. The support function $v_t$ for $X_t$
is given by
$$v_t(x)=-\langle X_t(x), x\rangle.$$
We know that $$(A_t(X))(A_tx)=A_t(X(x)).$$ Thus we get
\begin{eqnarray}\label{vt}
\begin{aligned}
v_t(x)=&-\langle A_t(X)(x), x\rangle \\
=&-\langle A_t(X)(A_tA_t^{-1}(x)), A_tA_t^{-1}x\rangle\nonumber\\
=&-\langle A_t(X(A_t^{-1}(x))), A_t(A_t^{-1}x)\rangle\nonumber\\
=&-\langle X(A_t^{-1}(x)), A_t^{-1}(x)\rangle \nonumber\\
=&\;v(A_t^{-1}(x))\nonumber.
\end{aligned}
\end{eqnarray}
Suppose $\{e_1,\cdots,e_n\}$ is an orthonormal frame of $\mathbb{S}^n$. Then we have
\begin{align}
\nabla_iv_t(x)&=dv_t(x)(e_i)=dv(A_t^{-1}x)(A_t^{-1}e_i),\\
\nabla^2_{ij}v_t(x)&=e_j((v_t)_i)-\nabla_{e_j}e_i v_t\label{D2v}\\
&=d[dv(A_t^{-1}x)(A_t^{-1}e_i)](A_t^{-1}e_j)-dv_t(x)(\nabla_{e_j}e_i)\nonumber\\
&=\nabla^2v(A_t^{-1}x)(A_t^{-1}e_i,A_t^{-1}e_j)+dv(A_t^{-1}x)(\nabla_{A_t^{-1}e_j}A_t^{-1}e_i)\nonumber\\
&-dv(A_t^{-1}x)(A_t^{-1}\nabla_{e_j}e_i)\nonumber\\
&=\nabla^2v(A_t^{-1}x)(A_t^{-1}e_i, A_t^{-1}e_j)\nonumber.
\end{align}
Here $\nabla$ is the canonical connection on $\mathbb{S}^n$ and we have used the fact that $$A_t^{-1}\nabla_{e_j}e_i=\nabla_{A_t^{-1}e_j}(A_t^{-1}e_i).$$
Note that, in the above equality,  the left hand side  represents the derivative taking at the point $x$
and the right hand side represents the derivative taking at the point $A_t^{-1} x$.

	Since $v$ satisfies equation \eqref{main equation hyperbolic},
using \eqref{vt} and \eqref{D2v}, we obtain
	\[
	F^*((v_t)_{i,j} + v_t\delta_{i,j})(x) = \tilde{\psi}(A_t^{-1}x, v_t(x)).
	\]
By \eqref{ust} and \eqref{vt}, we derive
	\
	\begin{align}
		\label{uts}
		\tilde{u}^*_t(y)&=\frac{\sqrt{1+|y|^2}}{\sqrt{1+|\sigma_t y|^2}}u^*_t(y)=\frac{\sqrt{1+|y|^2}}{\sqrt{1+|PA_t^{-1}P^{-1}y|^2}}u^*(PA_t^{-1}P^{-1}y)\\	
		&=\sqrt{1+|y|^2} v(A_t^{-1}P^{-1}y)=\sqrt{1+|y|^2} \, v_t \left(P^{-1} y\right).
		\nonumber
	\end{align}

Applying Lemma \ref{4.1}, we have
	\[
	\sum_{k,l}w^* \, b_{ik}^* (\tilde{u}^*_t)_{kl} \, b_{lj}^* = (v_t)_{i,j} + v_t \delta_{i,j},
	\]
which implies \eqref{Fst}.
		
\end{proof}
	
Suppose the components of $\sigma_t$ are
	\begin{equation}\label{sigmacom}
	\sigma_{t}(y) = (g_1(t,y), \ldots, g_n(t,y)).
	\end{equation}	
Now taking once and twice derivatives with respect to $t$ in \eqref{Fst}, we get
	\begin{equation}\label{D1}
	\sum_{k,l}F^{*ij} w^* \, b_{ik}^* \left( \frac{d}{dt} \tilde{u}_t^* \right)_{kl} b_{lj}^* = \sum_m\psi^*_{y_m}\frac{\p g_m}{\p t}+\psi^*_{u^*} \frac{d}{dt} u_t^* ,
	\end{equation}
and
\begin{align}\label{D2}
		&\sum_{k,l}F^{*ij} w^* b_{ik}^* \left( \frac{d^2}{dt^2} \tilde{u}_t^* \right)_{kl} b_{lj}^* \\
	&= -F^{*ij, rs} \left[\sum_{k,l}w^* b_{ik}^* \left( \frac{d}{dt} \tilde{u}_t^* \right)_{kl} b_{lj}^*\right]\left[ \sum_{p,q}w^* b_{rp}^* \left( \frac{d}{dt} \tilde{u}_t^* \right)_{pq} b_{qs}^* \right]+\psi^*_{u^*u^*}\left(\frac{d u_t^*}{dt}\right)^2\nonumber\\
	&\qquad+\sum_m\psi^*_{y_m}\frac{\p^2 g_m}{\p t^2}+\sum_{m,n}\psi^*_{y_my_n}\frac{\p g_m}{\p t}\frac{\p g_n}{\p t}+2\sum_m\psi^*_{y_m}\psi^*_{u^*}\frac{\p g_m}{\p t} \frac{d}{dt} u_t^*
	+\psi^*_{u^*} \frac{d^2}{dt^2} u_t^* .
	\nonumber	
\end{align}

Set $T$ to be the tangential vector field generated by $\sigma_t$, namely,
	\begin{equation}\label{T}
	T = \sum_{m=1}^n \left. \frac{\partial g_m}{\partial t} \right|_{t=0}\frac{\partial}{\partial y_m}.
	\end{equation}
To proceed, we need the following proposition.
\begin{prop}\label{prop}
Suppose $\{\sigma_t, t\in[0,\epsilon_0]\}$ is an one-parameter transformation group of $\mathbb{R}^n$, namely, for any $0\leq s,t\leq \epsilon_0$,
\begin{equation*}
\sigma_{t+s} = \sigma_t \circ \sigma_s = \sigma_s \circ \sigma_t, \text{ and } \sigma_{0}=id.
\end{equation*}
 For any smooth function $h(y)$, we define $h_t(y)=h(\sigma_t(y))$. Then
 we have
\begin{equation}\label{prop1}
\left. \frac{d}{dt} h_t  \right|_{t=0}(y) = T h(y) \ \  \text{ and }  \ \   \left. \frac{d^2}{dt^2} h_t \right|_{t=0} (y) = T^2 h (y),\end{equation}
where $T$ is defined by \eqref{sigmacom} and \eqref{T}.
\end{prop}

\begin{proof}
The first equality of \eqref{prop1} is obvious. To prove the second equality of \eqref{prop1}, we only need to show that
\[
\left. \frac{d^2}{dt^2} h_t \right|_{t=0} (y) = \left. \frac{d}{dt} \left[ T h (\sigma_t (y)) \right] \right|_{t=0},
\]
and use the first equality. Since
\[
\left. \frac{d^2}{dt^2} h_t \right|_{t=0} (y) =\lim_{t \to 0} \frac{1}{t} \left[ \frac{d}{dt} h_t (y) - \left. \frac{d}{dt} h_t (y) \right|_{t=0} \right]\]
and
\[ \left. \frac{d}{dt} \left( T h (\sigma_t (y)) \right) \right|_{t=0} = \lim_{t \to 0} \frac{1}{t} \left[ T h (\sigma_t (y)) - T h (y) \right],
\]
it suffices to prove
\[
\frac{d}{dt} h_t (y)= T h (\sigma_t (y)).
\]
Indeed we have
\begin{align*}
	\frac{d}{dt} h_t (y)&=\lim_{s \to 0} \frac{1}{s} \left[ h (\sigma_t(\sigma_s(y))) - h (\sigma_t(y)) \right] \\
&= \lim_{s \to 0} \frac{1}{s} \left[ h (\sigma_s (\sigma_t (y))) - h (\sigma_t (y)) \right]\\
&=\left. \frac{d}{ds} h_s (\sigma_t (y)) \right|_{s=0} \\
&= T h (\sigma_t (y)).	
\end{align*}

\end{proof}
Using \eqref{prop1}, we get	
	\begin{align}\label{newut}
	\left.\frac{d}{dt} \tilde{u}_t^*\right|_{t=0} &=w^*\left.\frac{d}{dt} \frac{u_t^*}{\sqrt{1+|\sigma_t(y)|^2}}\right|_{t=0}=w^*T\frac{u^*}{w^*},\\	
	\left.\frac{d^2}{dt^2} \tilde{u}_t^*\right|_{t=0} &=w^*\left.\frac{d^2}{dt^2} \frac{u_t^*}{\sqrt{1+|\sigma_t(y)|^2}}\right|_{t=0}=w^*T^2\frac{u^*}{w^*}.\nonumber
        \end{align}
Thus, letting $t=0$ in \eqref{D1}, \eqref{D2}, then using \eqref{newut}, Proposition \ref{prop} and the concavity of $F^*$, we obtain the following lemma:
\begin{lem}\label{lem4.3}
Suppose $u^*$ is the solution of \eqref{Legendre}. Let $T$ be a vector field generated by some one parameter transformation group $\{\sigma_t, t\in[0,\epsilon_0]\}$ of $\mathbb{R}^n$. Namely  $T$ is defined by \eqref{sigmacom} and \eqref{T}.
Then we have
	\begin{align}
\label{T2}
		\\
		\sum_{k,l}F^{*ij} w^* \, b_{ik}^* \left(w^*T \frac{u^*}{w^*}\right)_{kl} b_{lj}^*&= T\psi^*+\psi^*_{u^*}Tu^*, \nonumber\\
		\sum_{k,l}F^{*ij} w^* \, b_{ik}^* \left(w^*T^2 \frac{u^*}{w^*}\right)_{kl} b_{lj}^* &\geq  T^2\psi^*+2\psi^*_{u^*}T\psi^*Tu^*+\psi^*_{u^*}T^2u^*+\psi^{*}_{u^*u^*}(Tu^*)^2.\nonumber
	\end{align}

\end{lem}

Now we construct a ``canonical" vector field generated by rotations. Let $dP$ be the differential of the map $P$. We would like to construct a vector field near any given boundary point $y_0$, which parallel any given tangential vector at $y_0$.
\begin{prop}\label{propT}
Given some boundary point $y_0\in\p\Omega^*$ and some tangential vector $\xi=(\xi_1,\cdots,\xi_n)$ of $\p\Omega^*$ at $y_0$. We can construct a vector field $T=\sum_mT_m\frac{\p}{\p y_m}$ near $y_0$. It satisfies
\begin{equation}\label{5.2}
T_m(y_0)=\sqrt{1+|y_0|^2}\xi_m
\end{equation}
and
\begin{equation}\label{5.21}
|T(y)|^2=\sum_mT_m^2(y)\leq 1+|y|^2 \text{ for any } y.
\end{equation}
Here the equality of \eqref{5.21} holds when $y=y_0$.
Moreover,  \eqref{T2} holds for the special $T$ constructed here.
\end{prop}

\begin{proof}
Let $x_0 = P^{-1}(y_0)$ and $e_1$ be the direction of $(dP)^{-1}(\xi)$.
We expand $e_1$ to some orthonormal vectors $\{e_1, e_2, \ldots, e_{n-1}, e_n\}$ such that $\{e_1,\cdots,e_{n-1}\}$ can span the tangential space of $\p\tilde{\Omega}$ at $x_0$.
For any $x\in \mathbb{S}^n_+$, suppose
$$x= a_0(x) x_0 + \sum_{i=1}^n a_i(x)e_i, $$
where $a_0(x),a_1(x),\cdots, a_n(x)$ are the components of $x$.

Given constant $\epsilon_0>0$, for any $0\leq t\leq \epsilon_0$, we define a family of  rotations:
\begin{equation*}
	\begin{aligned}
		A_t (x_0) &= \cos t \, x_0 + \sin t \, e_1, \\
		A_t (e_1) &= -\sin t \, x_0 + \cos t \, e_1, \\
		A_t (e_i) &= e_i, \quad \text{for } i \geq 2.
	\end{aligned}
\end{equation*}
Then, we get
\begin{equation*}
	A_t (x) = a_0 A_t (x_0) + \sum_{i=1}^n a_i A_t (e_i).
\end{equation*}
Next we let $$\sigma_t (y) = P A_t P^{-1} (y) = (g_1(t,y), \cdots, g_n(t,y)),$$
and define the vector field
\begin{equation*}
	T = \sum_{i=1}^n \frac{\partial g_m}{\partial t} \Bigg|_{t=0} \frac{\partial}{\partial y_m}.
\end{equation*}
Suppose $x=P^{-1} (y) = (\tilde{x}, x_{n+1})$, then we see that
\begin{equation*}
	x_{n+1} =\frac{1}{\sqrt{1+\lvert y \rvert^2}}, \quad
	\tilde{x} =- \frac{y}{\sqrt{1+\lvert y \rvert^2}}.
\end{equation*}
Thus we have
\begin{equation*}
	a_0 (y) = \langle P^{-1} (y), x_0\rangle, \quad a_i (y) = \langle P^{-1} (y), e_i\rangle.
\end{equation*}
Therefore
\begin{equation*}
	A_t  P^{-1} (y) = (a_0(y) \cos t - a_1(y) \sin t) \, x_0 + (a_0(y) \sin t + a_1(y) \cos t) \, e_1 + \sum_{i=2}^n a_i(y) e_i.
\end{equation*}
Suppose $(y_1,\cdots,y_n)$ are the rectangular coordinate. We let
$E_i=\frac{\p}{\p y_i}$ for $i=1,\cdots,n+1$.
Thus we get
\begin{equation*}
	g_m (t,y) = -\frac{\langle A_t \, P^{-1} (y), E_m\rangle}{\langle A_t \, P^{-1} (y), E_{n+1}\rangle}.
\end{equation*}
Then, taking derivative with respect to $t$,  we have
\begin{equation*}
	\frac{d}{dt} A_t \, P^{-1} (y) \Bigg|_{t=0} = - a_1 (y) \, x_0 + a_0 (y) e_1
\end{equation*}
and
\begin{align}\label{5.4}
	&T_m(y) =\left. \frac{\partial g_m}{\partial t} \right|_{t=0} \\
	&= -\frac{\left.\frac{d}{dt} \langle A_t P^{-1}(y), E_m\rangle \right|_{t=0}}{\langle P^{-1}(y), E_{n+1}\rangle}
	 +\frac{\langle A_t P^{-1}(y),E_m\rangle}{ \langle A_t P^{-1}(y), E_{n+1} \rangle^2}
	\left. \frac{d}{dt} \langle A_t P^{-1}(y), E_{n+1}\rangle \right|_{t=0}\nonumber\\
&	= 
	\frac{a_1(y) \, \langle x_0,E_m\rangle - a_0(y) \, \langle e_1,E_m\rangle}{\frac{1 }{\sqrt{1 + |y|^2}} }  -y_m\sqrt{1 + |y|^2}
	\left[ -a_1(y) \, \langle x_0, E_{n+1}\rangle + a_0(y) \, \langle e_1, E_{n+1}\rangle \right]\nonumber\\
&=
\sqrt{1 + |y|^2} \left\{
\langle P^{-1}(y), e_1\rangle
( \langle x_0, E_m\rangle + y_m \langle x_0, E_{n+1}\rangle \right)\nonumber\\
&\qquad\qquad  -\langle P^{-1}(y), x_0\rangle \left( \langle e_1, E_m\rangle + y_m \langle e_1, E_{n+1} \rangle)
\right\}.\nonumber
\end{align}
Indeed it is a polynomial about  $y$ with degree 2.

It is clear that
\[
dP^{-1} \left( \frac{\partial}{\partial y_k} \right) = \frac{\p}{\p y_k} P^{-1}=
-\frac{1}{\sqrt{1 + |y|^2}} E_k - \frac{ y_k}{1 + |y|^2} P^{-1}(y).
\]
Then we can calculate
\[
\langle P^{-1}(y), dP^{-1}\left( \frac{\partial}{\partial y_k} \right) \rangle=  \frac{y_k}{1+|y|^2}  - \frac{y_k}{1+|y|^2}  = 0.
\]
Next we have
\begin{eqnarray*}
	\langle dP^{-1} \left( \frac{\partial}{\partial y_k} \right), ( E_m + y_m E_{n+1})\rangle = -\frac{1}{\sqrt{1+|y|^2}}  \delta_{mk}.
	\end{eqnarray*}
Letting $\tilde{\xi}=(\tilde{\xi}_1,\cdots,\tilde{\xi}_n) =dP(e_1)$ and using \eqref{5.4} and $\langle x_0, e_1\rangle=0$, we have

\begin{align}\label{newTm}
	T_m(y_0) &= -\sqrt{1+|y_0|^2} \left( \langle e_1,E_m\rangle + (y_0)_m \langle e_1,E_{n+1}\rangle \right)\\
	&= -\sqrt{1+|y_0|^2} \langle dP^{-1} (\tilde{\xi}), \left( E_m + (y_0)_m E_{n+1} \right)\rangle \nonumber\\
	&= -\sqrt{1+|y_0|^2} \sum_k \tilde{\xi}_k \langle dP^{-1} \left( \frac{\partial}{\partial y_k} \right), \left( E_m + (y_0)_m E_{n+1} \right)\rangle \nonumber\\
	&= \sum_k \tilde{\xi}_k \delta_{mk}=\tilde{\xi}_m.\nonumber
\end{align}
Here we let $y_0=((y_0)_1,(y_0)_2,\cdots, (y_0)_n)$.

Let us first prove \eqref{5.21}. Since we have
\begin{equation}\label{4.101}
P^{-1}(y)=a_0(y)x_0+\sum_{i\geq 1}a_i(y)e_i,
\end{equation}
where $a_0(y_0)=1$ and $a_i(y_0)=0$ for $i\geq 1$.
Define the function
\[
\tilde{\eta} (y) = a_1(y) x_0 - a_0(y) e_1.
\]
Then, using \eqref{5.4},  we get
\begin{eqnarray*}
T_m(y) 
&=& \sqrt{1+|y|^2}\langle \tilde{\eta}(y), (E_m + y_m E_{n+1})\rangle.
\end{eqnarray*}
We denote
\[
\tilde{e}_m = E_m + y_mE_{n+1} - y_m P^{-1}(y).
\]
Then, in view of $\langle\tilde{\eta}, P^{-1}(y) \rangle= 0
$, we have
\[
\langle\tilde{\eta}(y),\tilde{e}_m \rangle= \langle\tilde{\eta}(y), (E_m + y_m E_{n+1})\rangle.
\]
This allows us to conclude that
\[
T_m(y) = \sqrt{1+|y|^2} \langle\tilde{\eta}(y), \tilde{e}_m\rangle.
\]
Using the fact that
$\langle P^{-1}(y),(E_m+y_mE_{n+1})\rangle=0,$
we find that
\[
\tilde{e}_m \cdot \tilde{e}_{k}
= \delta_{km} + y_k y_m  - y_m y_k = \delta_{km}.
\]
Thus we can derive
\begin{align}\label{TTT}
\sum_{m=1}^{n} T_m^2  &=\left(1+|y|^2 \right) \sum_{m=1}^{n} \langle\tilde{\eta}, \tilde{e}_m\rangle^2 = \left( 1+|y|^2 \right) |\tilde{\eta}|^2 \\
&= \left( 1+|y|^2 \right) (a_0^2(y) + a_1^2(y))\leq 1+|y|^2,\nonumber
\end{align}
where the equality holds when $y=y_0$. Here we have used $a_0^2(y)+a^2_1(y)\leq |P^{-1}(y)|^2=1$. This completes the proof of \eqref{5.21}.
Next, from \eqref{newTm} and noting that $\tilde{\xi}$ is parallel to  $\xi$ and $|\xi|=1$, we obtain \eqref{5.2}.

Finally, for any $0\leq t,s\leq \epsilon_0$, since $A_{t+s} = A_t \cdot A_s = A_s \cdot A_t$ and $A_0=I$, we get
\[
\sigma_{t+s} = \sigma_t \circ \sigma_s = \sigma_s \circ \sigma_t, \text{ and } \sigma_0=id.
\]
Therefore, $\{\sigma_t, t\in[0,\epsilon_0]\}$ is an one-parameter transformation group.
Thus, we get \eqref{T2}.

\end{proof}

\section{$C^2$ estimates}
In this section, we continue to use the notations in Section \ref{Pre}. The main result of this section is  the following proposition.
\begin{prop}Suppose $\Omega$, $\Omega^*$ are bounded strictly convex domains with smooth boundary in $\mathbb{R}^{n}$.
 If $u^*$ is a strictly convex solution to \eqref{Legendre} with $\Omega=Du^*(\Omega^*)$,  then we have the $C^2$ estimate
\begin{equation*}
|D^2u^*|\leq C,
\end{equation*}
where $C$ is a constant depending on $\Omega^*$, $\sup_{\Omega^*}|u^*|$ and $\sup_{\Omega^*}|Du^*|$.
\end{prop}
\subsection{The global $C^2$ estimate}

To establish the estimate in $\Omega^*$, we utilize the equation \eqref{main equation hyperbolic}. Let $\tilde{\Omega} = P^{-1}(\Omega^*)$.
Suppose $\{e_1,\cdots, e_n\}$ is an orthonormal frame on $\mathbb{S}^n_+$. The spherical Hessian $\Lambda_{ij} $ defined by \eqref{sphere} satisfies 	
	\[
	\nabla_k\Lambda_{ij} =\nabla_j \Lambda_{ik}
	\]
and
	\[
	\nabla_{jj}\Lambda_{ii} -\nabla_{ii} \Lambda_{jj} = -\Lambda_{jj} + \Lambda_{ii}.
	\]
We consider the following problem
\[
	\sup_{x\in \tilde{\Omega},\ \eta \in T_{x} \mathbb{S}^n, |\eta|=1} \Lambda_{\eta\eta}(x).
	\]
Without loss of generality, we can assume $x_0, e_1$ are the maximum point and direction of the above problem. Moreover we can assume
$\Lambda_{ij}$ is a diagonal matrix at $x_0$.
Consequently, we have
	\[
	\Lambda_{\eta \eta} \leq \Lambda_{11}(x_0).
	\]
Next we will consider the test function:
	\[
	\varphi = \Lambda_{11}.
	\]
By using \eqref{main equation hyperbolic}, we have
$$F^{*ii}\Lambda_{ii11}=\nabla_{11}\tilde{\psi}+2(\nabla_1\tilde{\psi})\tilde{\psi}_{v}\nabla_1v+\tilde{\psi}_{v}\nabla_{11}v+\tilde{\psi}_{vv}(\nabla_1v)^2.$$
Since $|X|^2=|x|^2+u^2=|\nabla v|^2+v^2$, the above equation gives
$$F^{*ii}\Lambda_{ii11}\geq \tilde{\psi}_{v}\Lambda_{11}-C,$$ where $C$ depends on $\Omega^*$, $\sup_{\Omega^*}|u^*|$ and $\sup_{\Omega^*}|Du^*|$. Then we have
\begin{align*}
		F^{*ij} \varphi_{ij} &= F^{*ii} \Lambda_{11ii} = F^{*ii} \Lambda_{ii11} + \sum_i F^{*ii} \Lambda_{11} - \sum_i F^{*ii} \Lambda_{ii}\\
	&\geq\Lambda_{11} \sum_i F^{*ii} +\tilde{\psi}_v\Lambda_{11}- C.
\end{align*}

By \eqref{tildepsi} and \eqref{Cond1}, we have
 $$\tilde{\psi}_v=-\frac{\psi_z}{\psi^2}\geq 0.$$
Combining with $\sum F^{*ii} \geq F^*=\tilde{\psi}$, we get
	\[
	\Lambda_{11}(x_0) \leq C.
	\]
Otherwise, $x_0$ is on $\partial \tilde{\Omega}$. Therefore, we obtain
\[
\sup_{\Omega^*} \sum_{k,l}w^* b_{ik}^* D_{kl}^2 u^* b_{jl}^* \leq C + \sup_{\partial \Omega^*} \sum_{k,l}w^* b_{ik}^* D_{kl}^2 u^* b_{lj}^*,
\]
which implies
\begin{equation}\label{5.1}
\sup_{\Omega^*} |D^2 u^*| \leq C \left( 1 + \sup_{\partial \Omega^*} |D^2 u^*| \right).
\end{equation}

\subsection{$C^2$ boundary estimate}
Suppose $h^*(Du^*)= 0$ is the uniformly  concave defining function of $\partial \Omega$. Let $\beta=(h^*_{p_1}, \ldots, h^*_{p_n})$. Denote $\nu$ as the interior unit normal of $\p\Omega$. Thus we have $\beta(y)=\nu(Du^*(y))$. By Lemma \ref{3.4}, it follows that  $\beta \cdot \nu^* > 0$, where $\nu^*$ is the interior unit normal of $\p\Omega^*$.

\textbf{Step 1}: The Tangential-Normal Estimates.

We first establish that
\begin{equation}\label{tn}
D_{\tau\beta} u^* = 0, \ \ \text{on}\ \  \partial \Omega^*,
\end{equation}
 where $\tau$ is any tangential vector field of $\partial \Omega^*$.

\textbf{Step 2}: The Double Normal Estimates.

Let $H^* = h^*(Du^*)$. We compute the derivatives:
\[
D_i H^* = \sum_kh^*_{p_k} D_{i k} u^* \quad \text{and} \quad D_{ij} H^* = \sum_{k}h^*_{p_k} D_{ij k} u^* +\sum_{k,l} h^*_{p_k p_l} u^*_{i k} u^*_{j l}.
\]
For any function $f$, we still denote
\begin{equation}\label{L*}
\mathcal{L^*} f = \sum_{k,l}F^{*ij} w^*\, b_{ik}^* f_{kl} b_{lj}^*.
\end{equation}
Taking derivative with respect to $\frac{\p}{\p y_k}$ in \eqref{Legendre}, we get
\begin{eqnarray*}
\begin{aligned}
\sum_{p,q}F^{*ij} w^* \, b_{i p}^* u^*_{pq k} b_{qj}^* + \frac{y_k}{w^*} \sum_{p,q}F^{*ij} b_{i p}^*u^*_{pq} b_{qj}^*+2 \sum_{p,q}F^{*ij} w^*  \frac{\partial b^*_{ip}}{\partial y_k}  u_{pq}^* b_{qj}^*= \psi^*_{y_k}+\psi^*_{u^*}u_k^*.
\end{aligned}
\end{eqnarray*}
Using \eqref{L*}, the above equation reduces to
\begin{eqnarray*}
 \psi^*_{y_k}+\psi^*_{u^*}u_k^* =\mathcal{L}^* u_k^* + \frac{y_k}{w^{*2}} \psi^*+ 2\sum_{p,q} F^{*ij} w^* \frac{\partial b^*_{ir}}{\partial y_k} ( b^*)^{rs} b^*_{s p}u^*_{pq} b_{qj}^*.
\end{eqnarray*}
It is clear that
\[
\frac{\partial b^*_{ir}}{\partial y_k} = \frac{y_i \delta_{k r} + y_{r} \delta_{i k}}{1 + w^*} - \frac{y_i y_r}{(1 + w^*)^2} \frac{y_k}{w^*}
\]
 and 
 \[ (b^*)^{ rs} = \delta_{rs} - \frac{y_{r} y_s}{w^* (1 + w^*)}.
\]
Thus, by using the convexity of $u^*$, we conclude that
\begin{equation}\label{uk}
|\mathcal{L}^* u^*_k| \leq C,
\end{equation}
in view of the fact that $|\sum_{p,q}F^{*ij} w^*  b^*_{s p}u^*_{pq} b_{qj}^*|$ is uniformly bound for $1\leq i,s\leq n$.
Then we have
\begin{equation}\label{L*H}
\mathcal{L}^* H^* =\sum_k h^*_{p_k} \mathcal{L}^* u^*_k + \sum_{k,l,p,q}h^*_{p_k p_l}F^{*ij} w^* b_{i p}^* u^*_{ pk} u^*_{q l} b_{qj}^*.
\end{equation}
We have the following identity
\[
\sum_{p,q}F^{*ij} w^* b_{i p}^* u^*_{ pk} u^*_{q l}b_{qj}^*=\sum_{p,q,s,n}F^{*ij} w^* b_{i p}^* u^*_{p s} b_{s t}^* (b^*)^{tk} (b^*)^{lm} b^*_{mn} u^*_{ nq} b_{ qj}^*.
\]
From this, we can derive
\begin{equation}\label{ukk}
\left| \sum_{k,l,p,q}h^*_{p_k p_l} F^{*ij} w^* b_{i p}^* u^*_{ pk} u^*_{q l} b_{q j}^* \right| \leq C \sum_i F^{*ii} \lambda_i^2,
\end{equation}
where $\lambda_1,\cdots,\lambda_n$ are the curvature radius of $M_u$, which also are eigenvalues of the matrix $(\sum_{k,l}w^* b_{ik}^* u^*_{k l} b_{lj}^*)$.
Consequently,  we obtain
\begin{equation}\label{5.6}
\left| \mathcal{L}^* H^*\right| \leq  C \left(1 + \sum_i F^{*ii} \lambda_i^2\right).
\end{equation}


By using \cite{Urbas2001}, we know that, for any $\epsilon>0$,
\[
\sum_i F^{*ii} \lambda_i^2 \leq \left( C(\epsilon) + \epsilon M \right) \sum_i F^{*ii},
\]
where $C(\epsilon)$ is a constant depending only on $\epsilon$ and
\[
M = \sup_{y\in\Omega^*}\{\lambda_1(y),\cdots,\lambda_n(y)\}.
\]
For any boundary point $\hat{y}\in\p\Omega^*$, as discussed  in Section 4, we can define a neighborhood  $\Omega^*_r$ of $\hat{y}$ and barrier function $\varphi^*$  using \eqref{Omega*} and \eqref{boundary*}.  For positive constants $B$,
we consider the following test function:
\[
\Phi = H^* -B \left( C(\epsilon) + \epsilon M \right) \varphi^*
\]
in  $\Omega^*_r$. Since $-D^2\varphi^*\geq CI$, by using \eqref{5.6}, we have
\begin{align*}
		\mathcal{L}^* \Phi &\geq \mathcal{L}^*H^* + BC \left( C(\epsilon) + \epsilon M \right) F^{*ij} w^* b_{ik}^* b_{kj}^*\\
	&\geq -C \left( C(\epsilon) + \epsilon M \right) \sum_i F^{*ii} + BC \left(C(\epsilon) + \epsilon M \right)\sum_i F^{*ii}\\
	&\geq 0.
\end{align*}
Moreover, if $r$ is sufficiently small, we have $\Phi\leq 0$ on $\p\Omega^*$ and $\Phi(y_0)=0$. It follows that $\Phi$ achieves its maximum value at $\hat{y}$. Consequently, at $\hat{y}$,  we get
\[
D_{\nu^*} H^* \leq C \left( C(\epsilon) + \epsilon M \right).
\]
Since $\hat{y}$ is arbitrary, we conclude that
\begin{equation}\label{nn}
0 \leq D_{\beta \beta} u^* \leq C(\epsilon) + \epsilon M \quad \text{on} \quad \partial \Omega^*.
\end{equation}

\textbf{Step 3}: The Double Tangential Estimates.

 We consider the maximization problem:
\[
\tilde{M}=\max_{y\in \partial \Omega^*, \eta \in T_{y}(\partial \Omega^*), |\eta|=1}(1 + |y|^2) D_{\eta\eta} u^* (y).
\]
Let $y_0\in\p\Omega^*$ and the unit tangential vector $\xi \in T_{y_0}(\partial \Omega^*)$  be the maximum point and direction of the above problem.
Then we have
\[
\tilde{M}=(1+|y_0|^2) D_{\xi\xi}u^*(y_0).
\]
Let $x_0=P^{-1}(y_0)$ and $e_1$ be the unit vector which has the same direction as $dP^{-1}(\xi)$.
Using Proposition \ref{propT}, we can construct a vector field $T$ around $y_0$ with $T(y_0)=\sqrt{1+|y_0|^2}\xi$. Moreover we have  $|T(y)|^2\leq 1+|y|^2$.  Let
\[
T = \tau(T) + \frac{\langle\nu^*,T\rangle}{\langle\beta,\nu^*\rangle} \beta \quad \text{on} \quad \partial \Omega^*,
\]
where $\tau(T)$ is the tangential component of $T$ on the tangential space of $\partial \Omega^*$ and $\nu^*$ is the unit interior normal of $\p\Omega^*$.
Next we can write
\[
\beta = \beta^t + \langle\beta,\nu^*\rangle\nu^*,
\]
where $\beta^t$ is the tangential component of $\beta$ on the tangential space of $\p\Omega^*$.
Then, we get
\[
\tau(T) = T - \langle\nu^*,T\rangle \nu^* - \frac{\langle \nu^*,T\rangle}{\langle\beta,\nu^*\rangle} \beta^t.
\]
Using \eqref{tn}, we have
\begin{equation}\label{B1}
D_{TT} u^* = D_{\tau(T) \tau(T)} u^* + \frac{\langle\nu^*,T\rangle^2}{\langle\beta,\nu^*\rangle^2} D_{\beta \beta} u^*.
\end{equation}
It is clear that
\[
|\tau(T)|^2 = |T|^2 + \langle\nu^*,T\rangle^2 + \left( \frac{\langle\nu^*,T\rangle}{\langle\beta,\nu^*\rangle} \right)^2 |\beta^t|^2 - 2 \langle\nu^*, T\rangle^2 - 2 \frac{\langle\nu^*, T\rangle}{\langle \beta,\nu^*\rangle} \langle T, \beta^t\rangle,
\]
which implies
 \[
|\tau(T)|^2 \leq |T|^2 + C \langle\nu^*, T\rangle^2 - 2 \langle\nu^*, T\rangle \frac{\langle\beta^t, T\rangle}{\langle\beta, \nu^*\rangle}.
\]
We denote $\tilde{T} = \frac{1}{\sqrt{1+|y|^2}}T$, then the above inequality becomes
\[
|\tau(T)|^2 \leq  \left(1+|y|^2\right) \left( |\tilde{T}|^2 + C \langle\nu^*, \tilde{T}\rangle^2 - 2 \langle\nu^*, \tilde{T}\rangle \frac{\langle\beta^t, \tilde{T}\rangle}{\langle\beta, \nu^*\rangle} \right).
\]
Let $\eta=\tau(T)/|\tau(T)|$ be the unit tangential vector field.
Using the above inequality, we get
\[
D_{\tau(T)\tau(T)} u^*(y) \leq  \left( |\tilde{T}|^2 + C \langle\nu^*, \tilde{T}\rangle^2 - 2 \langle\nu^*, \tilde{T}\rangle \frac{\langle\beta^t, \tilde{T}\rangle}{\langle\beta, \nu^*\rangle} \right) \left( 1 + |y|^2 \right)  D_{\eta \eta} u^*(y).
\]
Now, applying equations \eqref{nn}, \eqref{B1} and the definition of $\tilde{M}$, we have
\begin{eqnarray*}
D_{TT} u^* &\leq& \left( |\tilde{T}|^2 + C \langle\nu^*, \tilde{T}\rangle^2 - 2 \langle\nu^*, \tilde{T}\rangle \frac{\langle\beta^t, \tilde{T}\rangle}{\langle\beta, \nu^*\rangle}  \right) \tilde{M}+(C(\epsilon) +\epsilon M)\left( \frac{\langle\nu^*, T\rangle}{\langle\beta, \nu^*\rangle} \right)^2.
\end{eqnarray*}
It implies that on $\p\Omega^*$,
\begin{equation}\label{B2}
\frac{D_{TT} u^*}{\tilde{M}}\leq   |\tilde{T}|^2 + C \langle\nu^*, \tilde{T}\rangle^2 - 2 \langle\nu^*, \tilde{T}\rangle \frac{\langle\beta^t, \tilde{T}\rangle}{\langle\beta, \nu^*\rangle} + \frac{C}{\tilde{M}}(C(\epsilon) +\epsilon M)\langle\nu^*, \tilde{T}\rangle^2 .
\end{equation}
Here we have used the fact that $\langle\beta,\nu^*\rangle$ has positive lower bound on $\p\Omega^*$.
Next we can define a function near the tubular neighborhood of $\p\Omega^*$,
\begin{equation}\label{funw}
w = \frac{D_{TT} u^*}{\tilde{M}} + 2 \langle\nu^*, \tilde{T}\rangle \frac{\langle\beta^t, \tilde{T}\rangle}{\langle\beta, \nu^*\rangle}- AH^*,
\end{equation}
where $A$ is a positive constant to be determined.
 At the point $y=y_0$, since $T(y_0)=\sqrt{1+|y_0|^2}\xi$, according to the definition of $\tilde{M}$, we find
\[
\tilde{T} = \xi,    \quad \text{and} \quad D_{TT} u^* = \tilde{M},
\]
which implies
\begin{equation}\label{B21}
w(y_0) = 1.
\end{equation}
Here we have used $h^*=0$ on $\p\Omega^*$.

On the other hand, for any $y \in\p\Omega^*$ and $y\neq y_0$, by \eqref{B2}, we have
\begin{equation}\label{B3}
w \leq |\tilde{T}|^2 + C \left( 1 + \frac{C (\epsilon) + \epsilon M}{\tilde{M}} \right) \langle\nu^*, \tilde{T}\rangle^2.
\end{equation}
Since $|T(y)|^2\leq 1+|y|^2$, we have  $|\tilde{T}|^2 \leq 1$.  Moreover, since the function $\langle \nu^*,\tilde{T}\rangle$ is smooth, for $y\in\p\Omega^*$, we obtain
\begin{equation}\label{B4}
|\langle\nu^*, \tilde{T}\rangle(y)| = |\langle\nu^*, \tilde{T}\rangle(y) - \langle\nu^*, \tilde{T}\rangle (y_0)| \leq C |y-y_0| \quad \text{near} \quad y_0.
\end{equation}
Combining \eqref{B3} with \eqref{B4}, we conclude that
\begin{equation}\label{B5}
w \leq 1 + C \left( 1 + \frac{C (\epsilon) + \epsilon M}{\tilde{M}} \right) |y-y_0|^2.
\end{equation}
By \eqref{5.1}, \eqref{nn} and the definitions of $M$ and $\tilde{M}$, we have
\begin{equation}\label{MM}
M \leq C \left( C(\epsilon)+\epsilon M + \tilde{M} \right).
\end{equation}
Assuming $\tilde{M} \geq 1$
, we can derive from \eqref{B5} and the inequality above that, for any $y \in \partial \Omega^*$,
\begin{equation}\label{ww}
w \leq 1 + C_1 (\epsilon) \left| y - y_0 \right|^2 \quad \text{near} \,\  y_0,
\end{equation}
if $\epsilon$ is sufficiently small, where $C_1(\epsilon)$ is a constant only depending on $\epsilon$.

As in Section 4, we also can define a neighborhood  $\Omega^*_r$ of $y_0$ and barrier function $\varphi^*$ by using \eqref{Omega*} and \eqref{boundary*}. Then we define a test function:
\[
\Psi = w - C_1(\epsilon) |y - y_0|^2 - B_1(C(\epsilon) +\epsilon M)\varphi^*,
\]
where $B_1$ is some undermined positive constant.  Thus, by using \eqref{B21} and $\varphi^*(y_0)=0$, we have $$\Psi(y_0) = 1.$$
Moreover, since $\varphi^*\geq 0$ on $\p\Omega^*_r\cap \p\Omega^*$ for $r$ sufficiently small and  using \eqref{ww}, on $\p\Omega^*_r\cap \p\Omega^*$, $\Psi$ achieves its maximum value at $y_0$.
On $\p\Omega_r^*\cap\Omega^*$, $\varphi^*$ has a positive lower bound. Thus, if $B_1$ is chosen sufficiently large, we have $\Psi\leq 1$ on $\p\Omega_r^*\cap\Omega^*$.

Now let's calculate $\Psi$ in $\Omega^*_r$. Using the notation defined by \eqref{L*}, we have
\begin{equation}\label{L1}
\mathcal{L}^* (|y - y_0|^2) = 2 \sum_kF^{*ij} w^* b_{ik}^* b_{kj}^* \leq C \sum_i F^{*ii}, \quad
\text{and}\quad
\mathcal{L}^* (-\varphi) \geq \theta \sum_i F^{*ii},
\end{equation}
where $\theta$ is a positive constant. Next we define $\chi$ as follows: $$\chi = 2 \langle\nu^*, \tilde{T}\rangle \frac{\langle\beta^t, \tilde{T}\rangle}{\langle\beta, \nu^*\rangle}.$$ Then $\chi(y,p) $ is a smooth function depending on $y$ and $p=D u^*$. Thus, by \eqref{funw}, we have
\begin{equation}\label{L2}
\mathcal{L}^* w = \frac{1}{\tilde{M}} \mathcal{L}^*D_{TT} u^* +  \mathcal{L}^* \chi - A \mathcal{L}^* H^*.
\end{equation}
We will now compute each term in $\mathcal{L}^* w$
individually. It is clear that
\[
D_{TT} u^* = T^2 u^* - (D_T T) u^*.
\]
Moreover we note that
\begin{eqnarray*}
\begin{aligned}
w^*T\frac{u^*}{w^*}=&\;Tu^*-\frac{u^*}{w^*}Tw^* \text{ and } \\ w^*T^2\frac{u^*}{w^*}=&\;T^2u^*-2Tu^*\frac{Tw^*}{w^*}-u^*\frac{T^2w^*}{w^*}+2u^*\frac{(Tw^*)^2}{(w^*)^2}.
\end{aligned}
\end{eqnarray*}
Therefore, by using \eqref{T2} and the above two equalities, we get
\begin{align*}
	\mathcal{L}^* (D_{TT} u^*) &\geq \mathcal{L}^* \left(2Tu^*\frac{Tw^*}{w^*}+u^*\frac{T^2w^*}{w^*}-2u^*\frac{(Tw^*)^2}{(w^*)^2}-D_T T (u^*)\right)\\
	&\qquad +T^2\psi^*+2\psi^*_{u^*}T\psi^*Tu^*+\psi^*_{u^*}T^2u^*+\psi^{*}_{u^*u^*}(Tu^*)^2\\
	&\geq\mathcal{L}^* \left(2Tu^*\frac{Tw^*}{w^*}-D_T T (u^*)\right)+\psi^*_{u^*}T^2u^*-C\sum_iF^{*ii}-C.
\end{align*}
We assume
\[
D_T T-2 \frac{Tw^*}{w^*} T= \sum_k t_k \frac{\partial}{\partial y_k}.
\]
Thus we get
\begin{align*}
\mathcal{L}^*\left(D_TTu^*-2 \frac{Tw^*}{w^*} Tu^*\right)&=\mathcal{L}^* \left( \sum_{k} t_k u^*_k \right) \\
&= \sum_{k,p,q} F^{*ij} w^* b^*_{ip} (t_k u^*_k)_{pq} b^*_{qj}\\
&= \sum_{k}t_k \mathcal{L}^* u^*_k + 2 \sum_{k,p,q} F^{*ij} w^* b^*_{ip} D_pt_kD_{kq}u^* b_{qj}^* \\
&+ \sum_{k,p,q} F^{*ij} w^* b^*_{ip} (t_{k})_{pq} u^*_kb_{qj}^*\\
&\leq C + C \sum_i F^{*ii}+ 2 \sum_{k,p,q} F^{*ij} w^* b^*_{ip} D_pt_s (b^*)^{sr}(b^*)_{rk}D_{kq}u^* b_{qj}^*\\
&\leq C + C \sum_i F^{*ii},	
\end{align*}
where we have used \eqref{uk} and
$$\left|\sum_{k,q} F^{*ij} w^*(b^*)_{rk}D_{kq}u^* b_{qj}^*\right|\leq C$$
for any $1\leq i,r\leq n$.
Thus, by combining the above four formulae and $D_{TT}u^* >0$, we obtain
\begin{equation}\label{L3}
\mathcal{L}^* (D_{TT} u^*) \geq - C - C \sum_i F^{*ii}+\psi^*_{u^*}D_{TT}u^*\geq - C - C \sum_i F^{*ii}.
\end{equation}
Here we have used $\psi^*_{u^*}=-\frac{\psi_z}{\psi^2}\geq 0$ in view of \eqref{psi*} and \eqref{Cond1}.
A straightforward calculation shows
\begin{align*}
	\chi_{i} &= \chi_{y_i} + \sum_s\chi_{p_s} u^*_{si},\\
	\chi_{ij} &= \chi_{y_iy_j} +\sum_s \chi_{y_i p_s} u^*_{sj} + \sum_s\chi_{p_s y_j}u^*_{si}+ \sum_{s,t}\chi_{p_sp_t} u^*_{si}u^*_{tj}+\sum_s\chi_{p_s}u^*_{sij}.
\end{align*}
Then, using \eqref{uk} and \eqref{ukk}, we have
\begin{align}	\label{L4}
	\mathcal{L}^* \chi &\geq -C \sum_i F^{*ii} + 2 \sum_{k,l,s}F^{*ij} w^* b^*_{ik} \chi_{y_kp_s} u^*_{sl} b_{lj}^*\\
	&+ \sum_{s,t,k,l} \chi_{p_sp_t} F^{*ij} w^* b^*_{ik}  u^*_{sk} u^*_{tl}b^*_{lj} - C\nonumber\\
	& \geq-C \sum_i F^{*ii} - C-C \sum_{s,k,l} F^{*ij} w^* b^*_{ik}  u^*_{sk} u^*_{sl}b^*_{lj} .\nonumber
\end{align}
On the other hand,  by using \eqref{L*H}, \eqref{uk} and the concavity of $h^*$, we get
	\begin{equation}\label{L5}
\mathcal{L}^* (-H^*) \geq \hat{\theta} \sum_{s,k,l}F^{*ij} w^* b^*_{ik} u^*_{ks} u^*_{sl} b^*_{lj} - C,
	\end{equation}
where $\hat{\theta}$ is a small positive constant. Thus, if $A, B_1$ are sufficiently large, combining \eqref{L1}-\eqref{L5}, we obtain
	\[
\mathcal{L}^*  \Psi \geq 0.
	\]
Thus we deduce that $\Psi$ achieves its maximum value at $y_0$. Therefore we have
	\[
	D_\beta w(y_0) \leq C(C(\epsilon) + \epsilon M).
	\]
This leads to
	\[
	D_{TT\beta} u^*(y_0) \leq C(C(\epsilon) + \epsilon M) \tilde{M},
	\]
which implies
	\begin{equation}\label{LL1}
	D_{\xi\xi \beta} u^*(y_0) \leq C(C(\epsilon) + \epsilon M) \tilde{M},
	\end{equation}
in view of $T(y_0)=\sqrt{1+|y_0|^2}\xi$.
	
By the boundary condition $h^*(Du^*) = 0$, at $y_0$, we take twice derivatives with respect to $\xi$, then we get
	\begin{equation}\label{LL2}
 D_{\xi\xi \beta} u^* +  \sum_{k,l}h^*_{p_kp_l} D_{\xi k} u^* D_{\xi l} u^* +II(\xi,\xi)D_{\nu^* \beta}u^*=0,\end{equation}
 where $II$ is the second fundamental form of $\p\Omega^*$ with respect to $\nu^*$.
Combining \eqref{LL1} with \eqref{LL2}, we get
 \[ - h^*_{p_kp_l} D_{\xi k} u^* D_{\xi l} u^* \leq C(C(\epsilon) + \epsilon M) \tilde{M}.
	\]
Therefore, using the concavity of $h^*$, we have
\[
(D_{\xi\xi} u^*)^2 \leq C \left( C(\epsilon) + \epsilon M \right) \tilde{M}.
\]
According  to the definition  of $\tilde{M}$, we get
\[
\tilde{M} \leq C \left( C(\epsilon) + \epsilon M \right).
\]
Combining \eqref{MM} and choosing $\epsilon$ sufficiently small , we conclude that

\[
M\leq C(\epsilon) .
\]
This completes the proof.


\end{document}